\documentclass[letterpaper,11pt]{amsart}
\usepackage{hyperref}
\usepackage{amsmath,amssymb,amsthm,amsopn}
\usepackage{fullpage}
\usepackage{graphicx}

\DeclareMathOperator{\Sg}{Sg}

\DeclareMathOperator{\tr}{tr}

\begin{document}

\makeatletter
\newtheorem*{rep@theorem}{\rep@title}
\newcommand{\newreptheorem}[2]{%
\newenvironment{rep#1}[1]{%
 \def\rep@title{#2 \ref{##1}}%
 \begin{rep@theorem}}%
 {\end{rep@theorem}}}
\makeatother

\newtheorem{thm}{Theorem}
\newreptheorem{thm}{Theorem}
\newtheorem{prop}{Proposition}
\newtheorem{cor}{Corollary}
\newtheorem{lem}{Lemma}
\newreptheorem{lem}{Lemma}

\theoremstyle{definition}
\newtheorem{defn}{Definition}
\newtheorem{conj}{Conjecture}
\newtheorem{prob}{Problem}

\theoremstyle{remark}
\newtheorem{rem}{Remark}
\newtheorem{ex}{Example}

\newcommand{\Rho}{\mathrm{P}}
\newcommand{\cA}{\mathcal{A}}
\newcommand{\cB}{\mathcal{B}}
\newcommand{\cS}{\mathcal{S}}
\newcommand{\cM}{\mathcal{M}}
\newcommand{\cF}{\mathcal{F}}
\newcommand{\cG}{\mathcal{G}}
\newcommand{\cH}{\mathcal{H}}
\newcommand{\cP}{\mathcal{P}}
\newcommand{\cR}{\mathcal{R}}
\newcommand{\cV}{\mathcal{V}}
\newcommand{\RR}{\mathbb{R}}
\newcommand{\ZZ}{\mathbb{Z}}
\newcommand{\NN}{\mathbb{N}}
\newcommand{\QQ}{\mathbb{Q}}
\newcommand{\bA}{\mathbb{A}}
\newcommand{\bB}{\mathbb{B}}
\newcommand{\bC}{\mathbb{C}}
\newcommand{\bD}{\mathbb{D}}
\newcommand{\bE}{\mathbb{E}}
\newcommand{\bF}{\mathbb{F}}
\newcommand{\bI}{\mathbb{I}}
\newcommand{\bS}{\mathbb{S}}
\newcommand{\fA}{\mathbf{A}}
\newcommand{\fB}{\mathbf{B}}
\newcommand{\gk}{\kappa}
\newcommand{\gS}{\Sigma}
\newcommand{\gl}{\lambda}
\newcommand{\gt}{\theta}

\newcommand{\dotcup}{\ensuremath{\mathaccent\cdot\cup}}

\title{Chromatic numbers of directed hypergraphs with no ``bad'' cycles}
\author{Zarathustra Brady}
\address{Department of Mathematics \\ Massachusetts Institute of Technology \\ 77 Massachusetts Avenue \\ Building 2, Room 350B\\ Cambridge, MA 02139-7307}
\email{notzeb@mit.edu}
\maketitle

\begin{abstract} Imagine that you are handed a rule for determining whether a cycle in a digraph is ``good'' or ``bad'', based on which edges of the cycle are traversed in the forward direction and which edges are traversed in the backward direction. Can you then construct a digraph which avoids having any ``bad'' cycles, but has arbitrarily large chromatic number?

We answer this question when the rule is described in terms of a finite state machine. The proof relies on Ne{\v s}et{\v r}il and R\"odl's structural Ramsey theory of posets with a linear extension. As an application, we give a new proof of the Loop Lemma of Barto, Kozik, and Niven in the special case of bounded width algebras.
\end{abstract}

\section{Setup}

Notation: $[k]$ stands for the set $\{1, ..., k\}$, $\cP(S)$ is the power-set of $S$, and $\Delta_S$ is the diagonal of $S\times S$. If $S$ is a set and $n$ is a natural number, then $\binom{S}{n}$ is the set of $n$-element subsets of $S$. Also, $P_n$ is a directed path of length $n$, that is, the digraph $([n+1],E)$ with $E = \{(i,i+1) \mid 1 \le i \le n\}$.

\begin{defn} A \emph{directed hypergraph} of uniformity $k$ is a pair $(V,E)$ with $E \subseteq V^k$. The \emph{chromatic number} of a directed hypergraph is the chromatic number of the associated undirected hypergraph, that is, the least number $\chi$ such that there exists a function $f:V\rightarrow [\chi]$ such that for each edge $e \in E$, not all of $f(e_1), ..., f(e_k)$ are equal. We'll assume that no edge of $E$ has any two coordinates equal to avoid annoying technical details which end up not mattering.
\end{defn}

\begin{defn} A $k$-\emph{machine} $\cM$ is a tuple $\cM = (S,f,\cB)$ where $S$ is a finite set of \emph{states}, $f : S\times [k]^2 \rightarrow \cP(S)$ is a \emph{transition function}, and $\cB \subseteq S\times S$ is the set of \emph{bad transitions}. We say that the $k$-machine $\cM$ is \emph{deterministic} if the value of $f(s,(i,j))$ always has size at most one, and is empty for $i = j$. If $\cM$ is deterministic, we abuse notation and think of $f$ as a function $f : S\times ([k]^2\setminus \Delta_{[k]}) \rightarrow S \cup \{\emptyset\}$, and think of $\emptyset$ as a special ``accepting'' state.
\end{defn}

\begin{defn} A \emph{cycle} of a $k$-uniform directed hypergraph $\cH = (V,E)$ is a sequence $c = (v_0, e_1, v_1, ..., e_n, v_n)$ with $v_n = v_0$, and $v_{i-1}, v_i \in \{(e_i)_1, ..., (e_i)_k\}$ for each $i$. We define $|c| = n$, and we define the \emph{trace} of $c$ by $\tr(c,i) = (a,b)$ where $(e_i)_a = v_{i-1}, (e_i)_b = v_i$. We say that the cycle $c$ of $\cH$ is $\cM$-\emph{bad} if there is a sequence of states $s_0, ..., s_n \in S$ such that for each $i$ we have
\[
s_i \in f(s_{i-1}, \tr(c,i)),
\]
and such that
\[
(s_0,s_n) \in \cB.
\]
We say that $\cH$ is $\cM$-\emph{good} if $\cH$ has no $\cM$-bad cycles.
\end{defn}

\begin{prob}\label{main-prob} Given a $k$-machine $\cM$, determine whether there exist $\cM$-good $k$-uniform directed hypergraphs $\cH$ of arbitrarily large chromatic number.
\end{prob}

\begin{ex} Let $k = 2$, and consider the deterministic $2$-machine $\cM = (\{s,t,u,v\},f,\{s\}\times\{t,v\})$, with $f$ given by $f(s,(1,2)) = t, f(t,(1,2)) = t, f(t,(2,1)) = u, f(u,(1,2)) = v$, and all other values of $f$ are $\emptyset$. Then a directed graph $\cG$ is $\cM$-good if and only if $\cG$ is the Hasse diagram of a poset.

It's well-known that Hasse diagrams can have arbitrarily large chromatic number (\cite{blanche}, \cite{coloring-lattices}, \cite{ramsey-lattices}, \cite{hasse-eyebrows}). An explicit poset whose Hasse diagram has chromatic number $n$ is the poset $(\binom{[2^n]}{2},\preceq)$ with $\{a,b\} \preceq \{c,d\}$ when $\max(a,b) \le \min(c,d)$ \cite{hasse-explicit}.
\end{ex}

\section{Warm up: cycling $k$-machines}

\begin{defn} A \emph{cycling} $k$-\emph{machine} $\cM$ is a $k$-machine $(S,f,\cB)$ such that $\cB = \Delta_S$.
\end{defn}

In the context of cycling $k$-machines, we only consider a cycle $c$ to be $\cM$-bad if it has $|c| > 0$. It's easy to modify a cycling $k$-machine such that it handles cycles of length $0$ correctly (while at most doubling the number of states), but this makes the definition clunky.

\begin{defn} If $\cM = (S,f,\Delta_S)$ is a cycling $k$-machine, we say that $\prec$ is an $\cM$-\emph{compatible order} on $S\times [k]$ if it is a total order such that the induced orderings on $S\times \{i\}$ agree for all $1 \le i \le k$, and for each $(s,i), (t,j) \in S\times [k]$ such that $t \in f(s,(i,j))$, we have $(s,i) \prec (t,j)$.
\end{defn}

\begin{thm}\label{cycling} If $\cM = (S,f,\Delta_S)$ is a cycling $k$-machine, then there exist $\cM$-good $k$-uniform directed hypergraphs $\cH$ of arbitrarily large chromatic number if and only if there is an $\cM$-compatible order $\prec$ on $S \times [k]$. Furthermore, if the chromatic number is bounded then it is bounded by $|S|!$.
\end{thm}
\begin{proof} First we show the necessity. Let $\cH = (V,E)$ be an arbitrary $k$-uniform directed hypergraph. We define an auxiliary digraph $\cG$ with vertex set $V\times S$ and edge set given by
\[
\{((a,s),(b,t)) \mid \exists e \in \cH,\ i, j \in [k]\text{ s.t. }e_i = a, e_j = b, t \in f(s,(i,j))\}.
\]
Any directed cycle in $\cG$ corresponds to an $\cM$-bad cycle in $\cH$, and vice-versa. Therefore if $\cH$ is $\cM$-good, then $\cG$ is a directed acyclic graph, so there exists a total order $\prec$ on $\cG$ such that if $((a,s),(b,t))$ is an edge of $\cG$ then $(a,s) \prec (b,t)$. Color the vertex $v \in V$ by the induced ordering $\prec\mid_{\{v\}\times S}$. If the chromatic number of $\cH$ is greater than $|S|!$, then there must exist an edge $e \in E$ such that $e_1, ..., e_k$ all have the same induced orderings. We now define the ordering $\prec$ on $S\times [k]$ by $(s,i) \prec (t,j)$ if and only if $(e_i,s) \prec (e_j,t)$, and note that this is an $\cM$-compatible order on $S\times [k]$.

Now we show the sufficiency. Fix an $\cM$-compatible order $\prec$ on $S\times [k]$. We define $\cH = (V,E)$ by taking $V = \binom{\NN}{|S|}$, and defining $E$ by
\[
E = \{(\{a_{11}, ..., a_{1|S|}\}, ..., \{a_{k1}, ..., a_{k|S|}\}) \mid a_{is} < a_{jt} \iff (s,i) \prec (t,j)\}.
\]
It's easy to show that this $\cH$ is $\cM$-good (the auxiliary digraph $\cG$ has vertices corresponding to elements of vertices of $\cH$, with the correspondence determined by the restriction of $\prec$ to any $S\times \{i\}$, and every edge of $\cG$ is increasing under the total ordering from $\NN$). Finally, the chromatic number of $\cH$ is infinite by Ramsey's theorem for hypergraphs (if we color the $k$-subsets of $\NN$ by finitely many colors, then there is some subset $C$ of $\NN$ of size $k|S|$ such that $\binom{C}{k}$ is monochromatic, and there is an edge $e \in E$ with $\cup_{i=1}^k e_i = C$).
\end{proof}

\begin{ex}\label{bounded-counter} Consider the family of cycling $2$-machines $\cM_n$, with $\cM_n = (\{0, ..., n\},f,\Delta_{\{0, ..., n\}})$ and $f(i,(1,2)) = \min(i+1,n)$ and $f(i,(2,1)) = i-2$ if $i \ge 2$, $f(0,(2,1)) = f(1,(2,1)) = \emptyset$. For $n \ge 2$, any $\cM_n$-good digraph must be the Hasse diagram of a poset (but the converse is not true). We'll use Theorem \ref{cycling} to show that for each $n$, there is an $\cM_n$-good digraph of infinite chromatic number. We just have to construct an $\cM_n$-compatible order $\prec$ on $\{0, ..., n\} \times [2]$. We take the restriction of $\prec$ to $\{0, ..., n\}\times\{1\}$ to be the reverse of the usual ordering (and the same for $\{0, ..., n\} \times \{2\}$), and take $(i,1) \prec (j,2)$ if and only if $i > j-2$.
\end{ex}

The following type of digraph, parametrized by a real number $\alpha > 1$, acts like a limiting case of Example \ref{bounded-counter} in the case $\alpha = 2$.

\begin{defn} Let $\alpha > 1$. We say that a digraph is $\alpha$-\emph{balanced} if every cycle which has $k$ forward edges has strictly less than $\alpha k$ backwards edges.
\end{defn}

\begin{prop} A digraph is $2$-balanced if and only if it is $\cM_n$-good for every $n$, with $\cM_n$ defined as in Example \ref{bounded-counter}.
\end{prop}

\begin{thm} Any $\alpha$-balanced digraph $\cG = (V,E)$ has chromatic number at most $\lceil \alpha\rceil +1$.
\end{thm}
\begin{proof} Assume WLOG that $\alpha$ is a whole number and that $\cG$ is connected and finite. Pick some vertex $v_0 \in V$, and for every walk $w$ from $v_0$ to a vertex $v\in V$, we let $\ell(w)$ be the number of forward steps in $w$ minus $\alpha$ times the number of backward steps in $w$. For $v \in V$, we let $\ell(v)$ be the supremum of $\ell(w)$ over all walks $w$ from $v_0$ to $v$. To see that $\ell(v)$ is finite, note that for any walk $w$ containing a cycle, we can delete that cycle to get a walk $w'$ with the same endpoints such that $\ell(w') > \ell(w)$ (by the definition of an $\alpha$-balanced digraph), and that only finitely many of the walks in $\cG$ contain no cycles. Now for any edge $(a,b) \in E$, we have $\ell(b) \ge \ell(a) + 1$, and $\ell(a) \ge \ell(b) - \alpha$, by extending a walk to $a$ or $b$ by one forward or backward step, respectively, so
\[
\ell(a) + 1 \le \ell(b) \le \ell(a) + \alpha.
\]
In particular, we have
\[
(a,b) \in E \;\; \implies \;\; \ell(a) \not\equiv \ell(b) \pmod{\alpha+1},
\]
so coloring the vertices of $\cG$ according to the remainder of $\ell(v) \pmod{\alpha+1}$ finishes the proof.
\end{proof}

\subsection{Hardness of checking for a compatible ordering}

We would like to know how difficult it is to test whether a cycling $k$-machine has a compatible order. Our first result shows that if we allow the uniformity $k$ to vary, then this is NP-complete.

\begin{thm} Checking whether a given deterministic cycling $k$-machine $\cM$ has a compatible order is NP-complete if $k$ is allowed to vary.
\end{thm}
\begin{proof} We'll reduce from 3-SAT. Suppose we have an instance with variables $V$ and constraints $C$, take $S$ to be the set of literals, take $k = 3|C|$, and number the constraints as $C_1, C_2, ...$. For each constraint $C_i$, we will introduce just three transitions for $\cM$, and we will have all other transitions lead to $\emptyset$. Suppose that $C_i$ is the $\vee$ of the literals $a, b, c$, with negations $\overline{a}, \overline{b}, \overline{c}$. Then the transitions corresponding to $C_i$ are as follows:
\begin{align*}
f(a,(3i-2,3i-1)) &= \overline{b},\\
f(b,(3i-1,3i)) &= \overline{c},\\
f(c,(3i,3i-2)) &= \overline{a}.
\end{align*}
Now, if $\prec$ is an $\cM$-compatible order on $S\times [k]$, then not all three of the inequalities $\overline{a} \prec a$, $\overline{b} \prec b$, $\overline{c} \prec c$ can be true, since these together with the above three transitions imply a directed cycle of inequalities. Thus, if we define the value for the literal $a$ to be true iff $a \prec \overline{a}$, then any $\cM$-compatible order corresponds to a solution to our instance of 3-SAT. Conversely, given a solution to our 3-SAT instance, we can use it to first decide which of the inequalities $a \prec \overline{a}$ should hold, then extend this to an order on $S$, and finally extending this to an $\cM$-compatible order on $S\times [k]$ is straightforward.
\end{proof}

Surprisingly, when $k = 2$ (i.e., in the case of digraphs), testing for a compatible ordering is equivalent to testing whether $P_n$ is $\cM$-good for all $n$.

\begin{thm} If $\cM$ is a cycling $2$-machine, then there are $\cM$-good digraphs having arbitrarily large chromatic number if and only if the directed path $P_n$ is $\cM$-good for all $n$. This can be tested in polynomial time (even if $\cM$ is non-deterministic).
\end{thm}
\begin{proof} Let $\cM = (S,f,\Delta_S)$. Suppose that $\prec$ is any total ordering on $S\times [2]$. Then there is an order preserving map $\iota : (S\times [2], \prec) \hookrightarrow (\QQ, <)$. Thinking of this as a map $S \rightarrow \QQ^2$, we can associate an interval $I_s \subset \QQ$ to each element $s \in S$, with endpoints $\iota(s,1)$ and $\iota(s,2)$.

Suppose now that $\prec\mid_{S\times\{1\}}$ agrees with $\prec\mid_{S\times\{2\}}$. It's easy to check that there can't be any $s,t \in S$ with $I_s \subset I_t$. Therefore, by the fact that proper interval graphs are always unit interval graphs (\cite{indifference-graphs}, \cite{proper-interval-algorithmic}), we may assume without loss of generality that
\[
|\iota(s,2) - \iota(s,1)| = 1
\]
for all $s \in S$. Additionally, if $I_s$ overlaps with $I_t$ for any $s,t \in S$, then we can check that the endpoints of $I_s$ must be sorted in the same way as the endpoints of $I_t$. Thus, within any connected component of our unit interval graph, all the intervals must have their endpoints sorted the same way.

Now we associate a weighted digraph $\cG$ to $\cM$, as follows. For any $s,t \in S$ and any $i,j \in [2]$, if $t \in f(s,(i,j))$ then we draw an edge from $s$ to $t$ with weight $j-i$ in $\cG$ (note that $\cG$ might have multiple edges of different weights connecting a pair of vertices, and that some edges may have weight $0$). Note that if $\prec$ is $\cM$-compatible, then every strongly connected component (ignoring the weights) of $\cG$ must be mapped to a connected component of our unit interval graph. It's easy to see that an $\cM$-compatible order on $S\times [2]$ exists if and only if each strongly connected component $C$ of $\cG$ has an $\cM$-compatible order on $C\times [2]$ (since we can linearly order the strongly connected components of $\cG$), so we may assume without loss of generality that $\cG$ is strongly connected.

If we have $\iota(s,2) > \iota(s,1)$ for all $s \in S$, then $\prec$ will be $\cM$-compatible if and only if the system of inequalities
\[
\{x_s < x_t + w \mid (s,t)\text{ is an edge of }\cG\text{ having weight }w\}
\]
is solved by taking $x_s = \iota(s,1)$. If $\iota(s,2) < \iota(s,1)$ for all $s \in S$, then the inequalities above must be replaced with $x_s < x_t - w$.

We will show that there is an $\cM$-compatible order if and only if $\cG$ has no directed cycles of total weight $0$ (an efficient way to test this is given in \cite{scaling-shortest-path}). First, if there is such a cycle, then adding the inequalities corresponding to its edges we see that the system of inequalities above has no solution (regardless of which way the endpoints of each interval are sorted). Conversely, if there is no solution to the above system of inequalities for either choice of how the endpoints of the intervals are sorted, then there must be a positive linear combination of these inequalities that comes out to $0 < 0$. Since each inequality has exactly one variable on each side, we can decompose this linear combination into positive linear combinations corresponding to directed cycles of $\cG$, to see that $\cG$ must have a directed cycle $c_+$ with nonnegative total weight and a directed cycle $c_-$ of nonpositive total weight. Since we have assumed that $\cG$ is strongly connected, it isn't hard to show that in fact $\cG$ must have a directed cycle of total weight $0$ (by finding a suitable positive linear combination of $c_+$, $c_-$, and any directed cycle that connects $c_+$ to $c_-$).
\end{proof}

\begin{thm} Checking whether a given (non-deterministic) cycling $3$-machine $\cM$ has a compatible order is NP-complete.
\end{thm}
\begin{proof}[Proof sketch] We restrict to the case of cycling $3$-machines such that for each state $s \in S$, we have $s \in f(s,(1,2))$, so that our compatible order $\prec$ must satisfy $(s,1) \prec (s,2)$. Using the fact that proper interval graphs are always unit interval graphs as in the proof of the previous theorem, to any compatible $\prec$ we can associate an order preserving map $\iota : (S\times [3], \prec) \rightarrow (\QQ, <)$ such that $\iota(s,2) = \iota(s,1) + 1$ for all $s \in S$. Introduce variables $x_s$ with $x_s = \iota(s,1)$. Since
\[
x_s < x_t \iff \iota(s,3) < \iota(t,3)
\]
for compatible orders $\prec$, we can find an increasing function $u : (\QQ, <) \rightarrow (\QQ, <)$ such that
\[
\iota(s,3) = u(x_s).
\]
Thus, the existence of a compatible order is equivalent to the existence of rational numbers $x_s$ for $s \in S$ and an increasing function $u$ satisfying a system of inequalities where each side of each inequality is in one of the forms $x_s, x_s + 1,$ or $u(x_s)$. Our goal is to show that solving such a system (for the $x_s$s and the unknown function $u$) is NP-complete.

Using polynomially many auxiliary variables, we can also use inequalities of the form $x_s < x_t \pm n$, where $n$ is a natural number which is at most polynomially large. Our main gadget will be based on the following observation. Suppose that $a_1, ..., a_n, b$ satisfy the system
\begin{align*}
\forall i \le n-1, \;\;\; a_{i+1} &< a_i + 1,\\
a_1 &< a_n - (n-2),\\
\forall i, \;\;\; u(a_i) &< a_i,\\
b + 1 &< u(b).
\end{align*}
Then we must have $b \not\in [a_1-1,a_n]$: if $b \in [a_i - 1, a_i]$, then
\[
a_i \le b + 1 < u(b) \le u(a_i) < a_i,
\]
a contradiction. If we let $m, k$ be natural numbers and let $x,y$ be two more variables, and add the inequalities
\begin{align*}
y + m &< a_1,\\
a_n &< y+m+n,\\
x + k &< b + 1,\\
b &< x + k,
\end{align*}
to the above system, then we see that
\[
x-y \not\in [(m-k)+2, (m-k)+n-2].
\]
The strategy is to fix $m-k$ and take $m$ large enough that the interval $[a_1-1,a_n]$ will not be anywhere near any other variables, other than $b$, giving us a gadget that guarantees that the difference $x-y$ is not in a given interval with integer endpoints.

Now it is straightforward to find a reduction from 3-coloring. Given a graph $\cG = (V,E)$, we introduce variables $x_v$ corresponding to the vertices of $V$, and use the gadget described above to force
\[
x_v - x_w \in [-21,-19] \cup [-11,-9] \cup [-1,1] \cup [9,11] \cup [19,21]
\]
for all $v,w \in V$. For each edge $\{v,w\} \in E$, we use the above gadget to add the additional constraint $x_v - x_w \not\in [-2,2]$. Given a solution to the above system, if we color the vertex $v$ of $\cG$ based on the closest multiple of $10$ to $x_v - x_{v_0}$ for some fixed vertex $v_0$, we get a 3-coloring of $\cG$, and conversely from a 3-coloring of $\cG$ we can easily construct a solution to the above system.
\end{proof}

\section{The general case}

In the general case, it is technically convenient to require trivial cycles not to be $\cM$-bad (in particular, if a nontrivial $\cM$-good hypergraph exists, we must have $\cB \cap \Delta_S = \emptyset$).

\begin{defn} We define an \emph{order system} on a set $S$ to be a triple $(\sim, \preceq, \le)$ such that $\sim$ is an equivalence relation on $S$, $\preceq$ is a partial order on $S/\!\sim$, and $\le$ is an extension of $\preceq$ to a total order on $S/\!\sim$.
\end{defn}

\begin{defn} If $\cM = (S,f,\cB)$ is a $k$-machine, then we say that the order system $(\sim, \preceq, \le)$ on $S \times [k]$ is \emph{compatible} with $\cM$ if it satisfies the following three conditions:
\begin{itemize}
\item For any $(s,i), (t,j) \in S\times [k]$ with $t \in f(s,(i,j))$, we have $((s,i)/\!\sim) \preceq ((t,j)/\!\sim)$.

\item The induced order systems $(\sim, \preceq, \le)\!\mid_{S\times \{i\}}$ on $S$ are independent of $i$.

\item For any $s,t \in S$ with $(s/\!\sim) \preceq (t/\!\sim)$ in the induced order system on $S$, we have $(s,t) \not\in \cB$.
\end{itemize}
\end{defn}

\begin{thm}\label{general} If $\cM = (S,f,\cB)$ is a $k$-machine, then there exist $\cM$-good $k$-uniform directed hypergraphs $\cH$ of arbitrarily large chromatic number if and only if there is an order system $(\sim, \preceq, \le)$ on $S \times [k]$ which is compatible with $\cM$. Furthermore, if the chromatic number is bounded then it is bounded by the number of possible order systems on $S$.
\end{thm}
\begin{proof} First we show the necessity. Let $\cH = (V,E)$ be an arbitrary $k$-uniform directed hypergraph. As in the cycling case, we define an auxiliary digraph $\cG$ with vertex set $V\times S$ and edge set given by
\[
\{((a,s),(b,t)) \mid \exists e \in \cH,\ i, j \in [k]\text{ s.t. }e_i = a, e_j = b, t \in f(s,(i,j))\}.
\]
We define an equivalence relation $\sim$ on the vertex set of $\cG$ by partitioning $\cG$ into its strongly connected components. Define a partial order $\preceq$ on $\cG/\!\sim$ by $(u/\!\sim) \preceq (v/\!\sim)$ if there exists a directed path from $u$ to $v$ in $\cG$. Finally, extend the partial order $\preceq$ to a total order $\le$ on $\cG/\!\sim$. Note that $\cH$ is $\cM$-good if and only if, for any $v \in V$ and any $(s,t)$ with $((v,s)/\!\sim) \preceq ((v,t)/\!\sim)$, we have $(s,t) \not\in \cB$.

Color the vertex $v \in V$ by the induced order system $(\sim, \preceq, \le)\!\mid_{\{v\}\times S}$. If the chromatic number of $\cH$ is greater than the number of possible order systems on $S$, then there must exist an edge $e \in E$ such that $e_1, ..., e_k$ all have the same induced order systems. We now define the order system $(\sim,\preceq,\le)$ on $S\times [k]$ by $(s,i) \sim (t,j)$ if and only if $(e_i,s) \sim (e_j,t)$, and similary for $\preceq, \le$, and note that this order system is compatible with $\cM$.

Now we show the sufficiency. Fix an order system $(\sim,\preceq,\le)$ on $S\times [k]$ which is compatible with $\cM$. Let $A$ be the structure $(S/\!\sim, \preceq\mid_{S/\!\sim}, \le\mid_{S/\!\sim})$, and let $B$ be the structure $((S\times[k])/\!\sim, \preceq, \le)$, so $A,B$ are both partial orders with linear extensions. Let $A_i$ by the induced copy of $A$ in $B$ coming from $S\times\{i\}$. By structural Ramsey theory for posets with a linear extension (Theorem 4.9 of \cite{all-those-ramsey}), there exists a partial order with linear extension $C$ such that for every way of coloring the set of induced copies of $A$ in $C$ by finitely many colors, there exists an induced copy $B'$ of $B$ in $C$ such that all induced copies of $A$ in $B'$ are colored with the same color.

We define $\cH = (V,E)$ by taking $V$ to be the set of induced copies of $A$ in $C$, and defining $E$ to be the set of $k$-tuples $(A_1', ..., A_k')$ such that there is an induced copy $B'$ of $B$ such that the map $B \xrightarrow{\sim} B'$ takes $A_i$ to $A_i'$.

It's easy to show that this $\cH$ is $\cM$-good (the auxiliary digraph $\cG$ has an equivalence relation $\sim$ such that the vertices of $\cG/\!\sim$ correspond to the elements of the induced copies of $A$ in $C$, and all of the edges of $\cG/\!\sim$  are non-decreasing with respect the partial order $\preceq$), and the chromatic number of $\cH$ is infinite by the choice of $C$.
\end{proof}

\begin{ex} Consider the $2$-machine $\cM = (\{0,1\},f,\{(0,1)\})$, with $f(0,(1,2)) = f(0,(2,1)) = \{1\}, f(1,(1,2)) = \{0\}$, and $f(1,(2,1)) = \emptyset$. A digraph is $\cM$-good if and only if it has no odd cycles such that every even-numbered edge points in the same direction (in particular, every odd cycle of an $\cM$-good digraph must have length at least $7$). There is a unique order system $(\sim, \preceq, \le)$ on $\{0,1\}\times [2]$ which is compatible with $\cM$: $(0,1) < (1,1) \sim (0,2) < (1,2)$, $0$ is incomparable with $1$ in the induced $\preceq$ on $\{0,1\}$, and $(0,1) \prec (1,2)$.

We can unwind the proof of Theorem \ref{general} to construct an explicit $\cM$-good digraph with infinite chromatic number as follows. For our vertex set, we take the set of ordered pairs $(A,B)$ of finite subsets of $\NN$ such that $A \not\subseteq B$ and $B \not\subseteq A$. For edges we take pairs of vertices of the form $((A,B),(B,C))$ such that $A \subset C$. It's easy to check that this digraph is $\cM$-good. To see that it has infinite chromatic number, we apply structural Ramsey theory for posets with a linear extension and note that every finite poset has an induced copy inside the poset of finite subsets of $\NN$.
\end{ex}

\section{Application to constructing terms in bounded width algebras}

We follow the same general proof strategy as in Theorem 3.2 of \cite{optimal-maltsev}. Rather than $(2,3)$-consistency, we'll use the framework of $pq$-instances from \cite{slac} - this will allow us to both prove stronger results and simplify the argument.

\begin{defn} We let $\cR_n$ be the set of subdirect relations on the $n$-element set $\{x_1, ..., x_n\}$. For $R, S \in \cR_n$, we define $R\circ S$ to be $\{(a,c) \mid \exists b\ (a,b) \in R,\ (b,c) \in S\}$, and we define $R^-$ to be $\{(b,a) \mid (a,b) \in R\}$.
\end{defn}

\begin{defn} We say a set $\cS \subseteq \cR_n$ of subdirect relations is $pq$-\emph{compatible} if $\cS$ is closed under composition and reversal, and for any $P,Q\in \cS$ there exists $j \ge 0$ such that
\[
\Delta_{\{x_1, ..., x_n\}} \subseteq P\circ (Q\circ P)^{\circ j}.
\]
\end{defn}

\begin{defn} For any $pq$-compatible set of subdirect relations $\cS$, and any function $\pi : ([k]^2\setminus \Delta_{[k]}) \rightarrow \cR$, we define the deterministic $k$-machine $\cM_{\cS,\pi}$ to be $\cM_{\cS,\pi} = (\cR, f, \{\Delta_{\{x_1,...,x_n\}}\}\times(\cR\setminus\cS))$, where $f$ is defined by $f(R,(i,j)) = R\circ \pi(i,j)$.
\end{defn}

\begin{thm} Let $R \subseteq \{x_1, ..., x_n\}^k$ be subdirect, and define $\pi$ by $\pi(i,j) = \pi_{i,j}(R)$. For any $pq$-compatible set $\cS$, if there are $\cM_{\cS,\pi}$-good $k$-uniform directed hypergraphs of arbitrarily large chromatic number, then for any finite bounded width algebra $\bA$ there exists a diagonal element in $\Sg_{\bA}(R)$.
\end{thm}
\begin{proof} This follows from the definition of $\cM_{\cS,\pi}$, the definition of a $pq$-compatible set of relations, and Theorem A.2 of \cite{slac}.
\end{proof}

\begin{cor} Every finite bounded width algebra has a $4$-ary term $t$ which satisfies $t(x,x,y,z) \approx t(y,z,z,x)$.
\end{cor}
\begin{proof} Let
\[
R = \left\{\begin{bmatrix} x \\ y \end{bmatrix}, \begin{bmatrix} x \\ z \end{bmatrix}, \begin{bmatrix} y \\ z \end{bmatrix}, \begin{bmatrix} z \\ x \end{bmatrix}\right\},
\]
define $\pi$ by $\pi(i,j) = \pi_{i,j}(R)$, and let $\cS$ be the set of relations in the compositional semigroup generated by $R,R^-$ which correspond to words which either contain $R\circ R$ or $R^- \circ R^-$, or are equal to $(R\circ R^-)^{\circ j}$ or $(R^- \circ R)^{\circ j}$ for some $j \ge 0$. Since every element of $\cS$ contains some power of the cyclic permutation $(x\; y\; z)$, and since both
\[
R\circ R\circ R = \left\{\begin{bmatrix} x \\ x \end{bmatrix}, \begin{bmatrix} x \\ y \end{bmatrix}, \begin{bmatrix} x \\ z \end{bmatrix}, \begin{bmatrix} y \\ y \end{bmatrix}, \begin{bmatrix} y \\ z \end{bmatrix}, \begin{bmatrix} z \\ x \end{bmatrix}, \begin{bmatrix} z \\ z \end{bmatrix}\right\}
\]
and
\[
R^- \circ R\circ R\circ R^- = \left\{\begin{bmatrix} x \\ x \end{bmatrix}, \begin{bmatrix} x \\ y \end{bmatrix}, \begin{bmatrix} y \\ x \end{bmatrix}, \begin{bmatrix} y \\ y \end{bmatrix}, \begin{bmatrix} y \\ z \end{bmatrix}, \begin{bmatrix} z \\ x \end{bmatrix}, \begin{bmatrix} z \\ y \end{bmatrix}, \begin{bmatrix} z \\ z \end{bmatrix}\right\}
\]
contain two distinct powers of the cyclic permutation $(x\; y\; z)$, $\cS$ is $pq$-compatible.

The $\cM_{\cS,\pi}$-bad cycles in a digraph are now exactly the odd cycles which alternate between forward steps and backward steps (aside from a single vertex where they do not alternate), so we just have to construct a digraph $\cG = (V,E)$ which has no odd alternating cycles and has infinite chromatic number. To finish the proof, we take $V = \{(a,b) \in \NN^2 \mid a < b\}$ and $E = \{((a,b),(b,c)) \mid a < b < c\}$.
\end{proof}

We can generalize the previous result, to give a new proof of the ``Loop Lemma'' (Theorem 3.5 of \cite{cyclic}, originally proved in \cite{smooth-digraph-dichotomy}) in the case of bounded width algebras.

\begin{prop} If $R \subseteq \{x_1, ..., x_n\}^2$, when viewed as a digraph, is smooth, weakly connected, and has algebraic length $1$, then there exists a number $k$ such that for all $l, m \ge k$,
\[
(R^{\circ l}\circ R^{-\circ m})^{\circ k} = \{x_1, ..., x_n\}^2.
\]
\end{prop}

\begin{defn} Say that a digraph $\cG$ is $k$-\emph{unbalanced} if for every directed cycle $c$ of $\cG$, either $c$ has exactly as many forward edges as backward edges, or there are two contiguous, non-overlapping stretches of $c$ such that one stretch has at least $k$ more forward edges than backward edges and the other stretch has at least $k$ fewer forward edges than backward edges.
\end{defn}

\begin{thm} For every fixed $k$, there exist digraphs which are $k$-unbalanced and have arbitrarily large chromatic number.
\end{thm}
\begin{proof} Define the deterministic $2$-machine $\cM_k$ to be $\cM_k = (S, f, \{a_0\}\times (S\setminus \{a_0\}))$, with
\[
S = \{a_{-k}, ..., a_k, b_0, ..., b_k\}
\]
and $f$ given by
\begin{align*}
\forall -k \le i < k, \;\;\; f(a_i,(1,2)) &= a_{i+1},\\
\forall -k < i \le k, \;\;\; f(a_i,(2,1)) &= a_{i-1},\\
f(a_k,(1,2)) &= b_0,\\
\forall 0 \le i < k, \;\;\; f(b_i,(2,1)) &= b_{i+1},\\
\forall 0 < i \le k, \;\;\; f(b_i,(1,2)) &= b_{i-1},\\
f(b_0,(1,2)) &= b_0,
\end{align*}
and all other values of $f$ are $\emptyset$. It's easy to see that any $\cM_k$-good digraph is $k$-unbalanced. Thus, by Theorem \ref{general}, it suffices to exhibit an order system $(\sim, \preceq, \le)$ on $S\times [2]$ which is compatible with $\cM$.

The equivalence relation $\sim$ and the total order $\le$ are given by
\begin{align*}
(a_{-k},2) &< (a_{-k},1) \sim (a_{-k+1},2) < \cdots \sim (a_k,2) < (a_k,1) < (b_0,1)\\
 &< (b_0,2) \sim (b_1,1) < (b_1,2) \sim \cdots < (b_{k-1},2) \sim (b_k,1) < (b_k,2).
\end{align*}
The partial order $\preceq$ is a little bit more delicate. On $\{b_0, ..., b_k\}\times [2]/\!\sim$, $\preceq$ agrees with $\le$, while $\{a_{-k}, ..., a_k\}\times[2]/\!\sim$ forms a $\preceq$-antichain. Between the $a_i$s and the $b_j$s, we have
\[
(a_i,u) \prec (b_j,v) \iff i+j > k+u-v.
\]
In particular, in the induced partial order on $S$, $a_0$ is not comparable to any other element of $S$.
\end{proof}

\begin{cor} If $R \subseteq \{x_1, ..., x_n\}^2$ is smooth and has algebraic length $1$ when viewed as a digraph, and if $\bA$ is a finite bounded width algebra, then there is a diagonal element in $\Sg_{\bA}(R)$.
\end{cor}

\subsection*{Acknowledgement} I would have never explored this line of research if Jelena Jovanovi{\'{c}}, Petar Markovi{\'{c}}, Ralph McKenzie, and Matthew Moore hadn't first come up with the outrageous idea of using Ramsey-theoretic arguments to construct terms in bounded width algebras. I'd also like to thank Petar Markovi{\'{c}} and his students Vlado Uljarevi\'{c} and Samir Zahirovi\'{c} for being excellent hosts and for enjoyable discussions of this circle of ideas when I visited Novi Sad.

\bibliography{csp}

\begin{thebibliography}{10}

\bibitem{cyclic}
Libor Barto and Marcin Kozik.
\newblock Absorbing subalgebras, cyclic terms, and the constraint satisfaction
  problem.
\newblock {\em Log. Methods Comput. Sci.}, 8(1):1:07, 27, 2012.

\bibitem{smooth-digraph-dichotomy}
Libor Barto, Marcin Kozik, and Todd Niven.
\newblock The {CSP} dichotomy holds for digraphs with no sources and no sinks
  (a positive answer to a conjecture of {B}ang-{J}ensen and {H}ell).
\newblock {\em SIAM Journal on Computing}, 38(5):1782--1802, 2009.

\bibitem{coloring-lattices}
B\'ela Bollob\'as.
\newblock Colouring lattices.
\newblock {\em Algebra Universalis}, 7(3):313--314, 1977.

\bibitem{blanche}
Blanche Descartes.
\newblock A three colour problem.
\newblock {\em Eureka}, 9(21):24--25, 1947.

\bibitem{hasse-explicit}
Stefan Felsner, Jens Gustedt, Michel Morvan, and Jean-Xavier Rampon.
\newblock Constructing colorings for diagrams.
\newblock {\em Discrete Applied Mathematics}, 51(1):85 -- 93, 1994.

\bibitem{proper-interval-algorithmic}
Fr\'ed\'eric Gardi.
\newblock The {R}oberts characterization of proper and unit interval graphs.
\newblock {\em Discrete Math.}, 307(22):2906--2908, 2007.

\bibitem{scaling-shortest-path}
Andrew~V Goldberg.
\newblock Scaling algorithms for the shortest paths problem.
\newblock {\em SIAM Journal on Computing}, 24(3):494--504, 1995.

\bibitem{all-those-ramsey}
Jan Hubi{\v{c}}ka and Jaroslav Ne{\v{s}}et{\v{r}}il.
\newblock All those {R}amsey classes ({R}amsey classes with closures and
  forbidden homomorphisms).
\newblock {\em arXiv preprint arXiv:1606.07979}, 2016.

\bibitem{optimal-maltsev}
Jelena Jovanovi{\'{c}}, Petar Markovi{\'{c}}, Ralph McKenzie, and Matthew
  Moore.
\newblock Optimal strong mal'cev conditions for congruence
  meet-semidistributivity in locally finite varieties.
\newblock {\em Algebra universalis}, pages 1--21, 2016.

\bibitem{slac}
Marcin Kozik.
\newblock Weaker consistency notions for all the {CSP}s of bounded width.
\newblock {\em CoRR}, abs/1605.00565, 2016.

\bibitem{hasse-eyebrows}
Igor K{\v r}\'\i{\v z} and Jaroslav Ne{\v s}et{\v r}il.
\newblock Chromatic number of {H}asse diagrams, eyebrows and dimension.
\newblock {\em Order}, 8(1):41--48, 1991.

\bibitem{ramsey-lattices}
Jaroslav Ne{\v s}et{\v r}il and Vojt{\v e}ch R\"odl.
\newblock Combinatorial partitions of finite posets and lattices---{R}amsey
  lattices.
\newblock {\em Algebra Universalis}, 19(1):106--119, 1984.

\bibitem{indifference-graphs}
Fred~S. Roberts.
\newblock Indifference graphs.
\newblock In {\em Proof {T}echniques in {G}raph {T}heory ({P}roc. {S}econd
  {A}nn {A}rbor {G}raph {T}heory {C}onf., {A}nn {A}rbor, {M}ich., 1968)}, pages
  139--146. Academic Press, New York, 1969.

\end{thebibliography}
\bibliographystyle{plain}

\end{document}